\documentclass[11pt]{amsart}

\usepackage{amsthm}
\usepackage{amsmath}
\usepackage{amssymb}
\usepackage{color}

\newtheorem{thm}{Theorem}[section]
\newtheorem{theorem}[thm]{Theorem}

\newtheorem{corollary}[thm]{Corollary}

\newtheorem{proposition}[thm]{Proposition}

\newtheorem{definition}[thm]{Definition}
\theoremstyle{remark}

\newtheorem{ex}[thm]{Example}

\newcommand{\RR}{\mathbb R}
\newcommand{\NN}{\mathbb N}

\newcommand{\HH}{\mathbb H}

\begin{document}

\title{The exact constant for the $\ell_1-\ell_2$ norm inequality}
\author[Botelho-Andrade, Casazza, Cheng, Tran
 ]{Sara Botelho-Andrade, Peter G. Casazza, Desai Cheng, and Tin Tran}
\address{Department of Mathematics, University
of Missouri, Columbia, MO 65211-4100}

\thanks{The authors were supported by
 NSF DMS 1609760;  and ARO  W911NF-16-1-0008}

\email{sandrade102087@gmail.com, Casazzap@missouri.edu, chengdesai@yahoo.com, tinmizzou@gmail.com}


\begin{abstract}
A fundamental inequality for Hilbert spaces is the $\ell_1-\ell_2$-norm inequality which gives that for any $x \in \HH^n$,
$\|x\|_1\le \sqrt{n}\|x\|_2.$  But this is a strict inequality for all but vectors with constant modulus for their coefficients.
We will give a trivial method to compute, for each x, the constant $c$ for which $\|x\|_1=c\sqrt{n}\|x\|_2.$ 
Since this inequality is one of the most used results in Hilbert space
theory, we believe this
will have unlimited applications in the field.  We will also show some variations of this result.
%\marginpar
\end{abstract}

\maketitle

\section{Introduction}

The $\ell_1-\ell_2$-norm inequality which gives that for any $x\in \HH^n$, $\|x\|_1\le \sqrt{n}\|x\|_2$.  
But this is a strict inequality for all but vectors with constant modulus for their coefficients.
We will give a trivial method to compute, for each x, the constant $c$ for which $\|x\|_1=c\sqrt{n}|x\|_2.$ 
Since this is one of the most fundamental and most used inequalities in Hilbert space theory,
we believe this will have broad application in the field.  We will also show some variations of this result.  For a background in this
area see \cite{CL,CK}.

\section{The $\ell_1-\ell_2$-norm Inequality}

We need a definition.

\begin{definition}
A vector of the form $x=\frac{1}{\sqrt{n}}(c_1,c_2, \ldots,c_n)\in \HH^n$, with $|c_i|=1$ for all $i=1,2,\ldots,n$
will be called a {\bf constant modulus vector}.
\end{definition}

\begin{theorem}\label{TT1}
Let $x=(a_1,a_2,\ldots,a_n)\in \HH^n$, a real or complex Hilbert space.  The following are equivalent:
\begin{enumerate}
\item We have
\[ \|x\|_1 = (1-\frac{c_x}{2})\sqrt{n}\|x\|_2.\]
\item We have
\[ \sum_{i=1}^n \left |\frac{|a_i|}{\|x\|_2}-\frac{1}{\sqrt{n}}\right |^2 = c_x.\]
\item The infimum of the distance from $\frac{x}{\|x\|_2}$ to the constant modulus vectors
is $\sqrt{c_x}$.
\end{enumerate}
In particular, 
\[ \|x\|_1 \le \sqrt{s}\|x\|_2,\]
if and only if
\[ \left ( 1-\frac{c_x}{2}\right ) \sqrt{n}\le \sqrt{s},\]
if and only if
\[ 1-\frac{c_x}{2}\le \sqrt{\frac{s}{n}}.\]
\end{theorem}

\begin{proof}
$(1)\Leftrightarrow (2)$:
We compute:
\begin{align*}
\sum_{i=1}^n \left |\frac{|a_i|}{\|x\|_2}-\frac{1}{\sqrt{n}}\right |^2&= \frac{1}{\|x\|_2^2}\sum_{i=1}^n|a_i|^2
+ \sum_{i=1}^n\frac{1}{n}-\frac{2}{\sqrt{n}\|x\|_2}\sum_{i=1}^n|a_i|\\ 
&= 2\left ( 1-\frac{1}{\sqrt{n}\|x\|_2}\sum_{i=1}^n|a_i|\right )=c_x.
\end{align*}
if and only if
\[ \frac{1}{\sqrt{n}\|x\|_2}\sum_{i=1}^n|a_i|= 1-\frac{c_x}{2},\]
if and only if
\[ \sum_{i=1}^n|a_i|= (1-\frac{c_x}{2})\sqrt{n}\|x\|_2.\]
$(1)\Leftrightarrow (3)$:  We compute:
\[
\inf\left \{ \sum_{i=1}^{n}\left|\frac{a_i}{\|x\|_2}-\frac{c_i}{\sqrt{n}}\right |^2:
|c_i|=1\right \}=\]
\[
\inf\left \{\frac{1}{\|x\|^2_2}\sum_{i=1}^n|a_i|^2
+\sum_{i=1}^n\left |\frac{c_i}{\sqrt{n}}\right |^2-2\frac{1}{\|x\|_2\sqrt{n}}Re\sum_{i=1}^na_i\bar{c_i}: \vert c_i\vert=1
\right \}=\]
\[ 2-\frac{2}{\|x\|_2\sqrt{n}}\sum_{i=1}^n|a_i|.
\]
The equality occurs when $\frac{1}{\sqrt{n}}(c_1,c_2, \ldots,c_n)$ is a constant modulus vector with $c_i=\dfrac{a_i}{\vert a_i\vert}$ if $a_i\not=0$.

Thus
\[ c_x= 2-\frac{2}{\|x\|_2\sqrt{n}}\sum_{i=1}^n|a_i|\mbox{ if and
only if (1) holds} \]
\end{proof}

Now we want to look at an application of the above.
For this we need two preliminary results.  

\begin{theorem}\label{thm5}
Let $S$ be a subspace of $\HH^n$ and let $P$ be the orthogonal
projection on $S$.  For any $x\in \HH^n$, $\frac{Px}{\|Px\|}$
is the closest unit vector in $S$ to $x$.
\end{theorem}
\begin{proof}
	Let $y$ be a unit vector in $S$ and extend it to be an orthonormal basis $\{y, u_1, u_2, \ldots, u_k\}$ for $S$. Then 
	$$Px=\langle x, y\rangle y+\sum_{i=1}^{k}\langle x, u_i\rangle u_i.$$
	Hence
	$$\Vert Px\Vert^2=\vert \langle x, y\rangle\vert^2+\sum_{i=1}^{k}\vert\langle x, u_i\rangle\vert^2\geq \vert \langle x, y\rangle\vert^2.$$
	Therefore $$\Vert Px\Vert\geq \vert \langle x, y\rangle\vert \geq Re\langle x, y\rangle.$$
	Now we have
	$$\Vert x-\dfrac{Px}{\Vert Px\Vert}\Vert^2=\Vert x\Vert^2-2\Vert Px\Vert+1\leq \Vert x\Vert^2-2Re\langle x, y\rangle+\Vert y\Vert^2=\Vert x-y\Vert^2,$$ which is our claim.
\end{proof}

Next, we examine
the $\ell_1-\ell_2$-norm inequality for subspaces.

\begin{theorem}\label{thm2}
Let $S$ be a subspace of $\HH^n$ and let $P$
be the projection onto $S$.  The following are equivalent:
\begin{enumerate}
\item For every unit vector $x\in S$,
\[ \|x\|_1 \le (1-\frac{c}{2})\sqrt{n}.\]
\item The $\ell_2$ distance of any unit vector in S to any
constant modulus vector is greater than or equal to $\sqrt{c}$.
\item For every constant modulus vector $x$, we have
\[ \|Px\|_2\le 1-\frac{c}{2}.\]
\end{enumerate}
\end{theorem}

\begin{proof}
$(1)\Leftrightarrow (2)$:  Let $x=(a_1,a_2,\ldots,a_n)$.
\begin{align*}
 \inf\{\sum_{i=1}^n|a_i-\frac{c_i}{\sqrt{n}}|^2:\vert c_i\vert=1\}
&= \inf\{\sum_{i=1}^n|a_i|^2 + \sum_{i=1}^n\frac{1}{n}
-\frac{2}{\sqrt{n}}Re\sum_{i=1}^na_i\bar{c_i}: \vert c_i\vert=1 \}\\
&= 2-\frac{2}{\sqrt{n}}\sum_{i=1}^n|a_i|.
\end{align*}
Now, 
\[ c \le 2-\frac{2}{\sqrt{n}}\sum_{i=1}^n|a_i|\mbox{ if and only
if } \sum_{i=1}^n|a_i|\le \left ( 1-\frac{c}{2}\right )\sqrt{n}.\]

\vskip10pt
$(2) \Leftrightarrow (3)$:  By Theorem \ref{thm5}, we need to check
how close 
\[ \frac{Px}{\|Px\|} \mbox{ is to the all one's vector x.}\]
So we compute:
\begin{align*}
\left \|\frac{Px}{\|Px\|}-x\right \|^2= 2-\langle \dfrac{Px}{\Vert Px\Vert}, x\rangle - \langle x, \dfrac{Px}{\Vert Px\Vert}\rangle
= 2-2\|Px\|
\end{align*}
So, 
\[ c \le \left \|\frac{Px}{\|Px\|}-x\right \|^2\mbox{
if and only if }
\|Px\|\le 1-\frac{c}{2}.\]
\end{proof}

Now we have the second main result.  For this recall \cite{CL,CK}
that if $P$ is a projection on $\HH^n$ with orthonormal basis
$\{e_i\}_{i=1}^n$ then $\sum_{i=1}^n\|Pe_i\|^2=dim\ P(\HH^n)$.

\begin{theorem}\label{thm3}
Let $S$ be a s-dimensional subspace of $\RR^n$
with orthonormal basis $\{e_i\}_{i=1}^n$.  If
\[ \|y\|_1 \le \sqrt{s}\|y\|_2,\mbox{ for all }y\in S,\]
then there is an $I\subset [n]$ with $|I|=s$
and $S = span\ \{e_i\}_{i\in I}$.
\end{theorem}

\begin{proof}
For any $y\in S$, let $c_y$ be defined in (2) of Theorem \ref{TT1}. Since \[ \|y\|_1 \le \sqrt{s}\|y\|_2,\mbox{ for all }y\in S,\] then 
\begin{equation*} 
1-\frac{c_y}{2}\le \sqrt{\frac{s}{n}}.
\end{equation*}
Set $$c=\inf\{c_y: y\in S\}$$ then 
\begin{equation}\label{E1}
1-\frac{c}{2}\le \sqrt{\frac{s}{n}}
\end{equation}
We will prove:  $\{Pe_i\}_{i=1}^n$ is an orthogonal set.  This will
impliy that there is an $I\subset [n]$ so that $Pe_i=e_i$ for 
$i\in I$ and $Pe_i=0$ for $i\in I^c$.

First note that $\{Pe_i\}_{i=1}^n$ is a Parseval frame for
$S$ and so
\[ \sum_{i=1}^n\|Pe_i\|^2 = s.\] 
Assume there are two of these vectors which are not orthogonal.
By reindexing, we will assume $Pe_1,Pe_2$ are not orthogonal.
Hence, by replacing $Pe_2$ by $c_2Pe_2$ with $|c_2|=1$ if necessary 
with $Re\ c_2\langle Pe_1,Pe_2\rangle>0$, we have
\[ \|Pe_1+c_2Pe_2\|^2 > \|Pe_1\|^2 + \|Pe_2\|^2.\]
Now, by replacing $Pe_3$ by $c_3Pe_3$ 
with $|c_3|=1$ if necessary, we have
\[ \|Pe_1+c_2Pe_2+c_3Pe_3\|^2 \ge \|Pe_1+Pe_2\|^2 + \|Pe_3\|^2
> \|Pe_1\|^2+\|Pe_2\|^2+\|Pe_3\|^3.\]
Continuing, and letting $c_1=1$, we have
\[ \|P\left ( \sum_{i=1}^nc_i e_i\right )\|^2 > \sum_{i=1}^n\|Pe_i\|^2=s. \]
It follows from Theorem \ref{thm2},
\[ \sqrt{\frac{s}{n}} < \left \| P\left (\frac{1}{\sqrt{n}} 
\sum_{i=1}^n c_i e_i \right ) \right \|\le 1-\frac{c}{2},
\]
which contradicts Equation (\ref{E1}) above.
\end{proof}

\section{An Application to $L_p[0,1]$}

It was pointed out to us by Bill Johnson that our work has application to Banach space theory.  That, in general, when
working with finite dimensional $\ell_p$, it is better to use the
$L_p[0,1]$ normalization.  But applying our results, the
{\it nasty $n^{1/2}$ goes away} and the expressions are independent
of dimension.  What is quite interesting here is the fact that
if $p<s$ and $f\in L_1[0,1]$ then we can measure {\it how peaky
f is} by seeing how small $\|f\|_p$ is. What apparently was
not realized is that when $p=1$ and $s=2$ we get a nice equality
instead of an inquality.

\begin{theorem}
Let $f\ge 0$ be norm one in $L_2[0,1]$.  The following are
equivalent:
\begin{enumerate}
\item We have
\[ \|f\|_1 = (1-\frac{c}{2}).\]
\item We have
\[ \|f-1\|_2^2=c.\]
\end{enumerate}
\end{theorem}

\begin{proof}
We use the parallelogram law:
\begin{align*}
4 &= \|f-1\|_2^2 +\|f+1\|_2^2\\
&= \|f-1\|_2^2+ 
 \|f\|_2^2+1+2\int_0^1 f(t)dt\\
 &= \|f-1\|_2^2+2+2\|f\|_1.
 \end{align*}
 I.e.
 \[ \|f-1\|_2^2= 2-2\|f\|_1.\]
 It follows that
 \[ \|f-1\|_2^2 = c \mbox{ if and only if }\|f\|_1=1-\frac{c}{2}.\]
\end{proof}

\end{document}